\documentclass[10pt]{amsart}

\usepackage{amsmath,amssymb,mathtools}
\usepackage{booktabs}
\usepackage{microtype}
\usepackage[hidelinks]{hyperref}

\newtheorem{theorem}{Theorem}[section]
\newtheorem{proposition}[theorem]{Proposition}

\newtheorem{corollary}[theorem]{Corollary}
\theoremstyle{remark}

\newcommand{\Q}{\mathbf Q}
\newcommand{\Z}{\mathbf Z}
\newcommand{\ii}{\mathrm i}

\title[Equalities of simplest quartic fields]
{Equalities among simplest quartic fields: a complete classification}
\author[Zhi-Lin Zhang]{Zhi-Lin Zhang}

\address{Independent Researcher, Taipei, Taiwan}

\email{hsa00000@gmail.com}

\date{}


\begin{document}

\begin{abstract}
For a positive integer $n$, let
\[
 f_n(X)=X^4-nX^3-6X^2+nX+1
\]
and let $K_n=\Q(\rho_n)$, where $\rho_n$ is a root of $f_n$.
We determine all coincidences among these fields: for distinct positive
integers $m,n$,
\[
 K_m=K_n
 \quad\Longleftrightarrow\quad
 \{m,n\}\in
 \bigl\{\{1,103\},\{2,22\},\{4,956\}\bigr\}.
\]
Thus the three previously known equalities are the only ones.
This extends Hoshi's finite-range classification to all positive integral
parameters and, in particular, subsumes the uniqueness results of
Pincus and Washington.

The proof combines Hoshi's correspondence between equal simplest
quartic fields and primitive solutions of a quartic Thue equation with
estimates of Lettl--Peth\H{o}--Voutier for rational approximations to
two of its real roots.  A Gaussian-integer identity yields a lower bound
for the denominator of the resulting rational approximation; the
continued-fraction information and the approximation estimates then
exclude every parameter exceeding $1000$.
\end{abstract}

\maketitle

\section{Introduction}

For $n\in\Z_{>0}$ set
\begin{equation}\label{eq:fn}
 f_n(X)=X^4-nX^3-6X^2+nX+1
\end{equation}
and let $K_n=\Q(\rho_n)$, where $\rho_n$ is any root of $f_n$.  We also write
\begin{equation}\label{eq:Ft}
 F_t(X,Y)=X^4-tX^3Y-6X^2Y^2+tXY^3+Y^4,
\end{equation}
so that $f_t(X)=F_t(X,1)$.

This polynomial family was studied earlier by M.-N.~Gras.  She proved that,
for $t\in\Z\setminus\{0,\pm3\}$, the polynomial $f_t$ is irreducible over
$\Q$ and defines a real cyclic quartic extension
\cite[p.~16, Proposition~6]{Gras1978}.  The same work contains extensive
numerical computations of class numbers and units in real cyclic quartic
fields.  In its example with parameter $t=22$, Gras observed that a
generating relative unit in the same field corresponds to parameter $t=2$;
in the present notation, this realizes the equality $K_2=K_{22}$
\cite[p.~18]{Gras1978}.

The field-isomorphism problem for this family was later studied by Hoshi, who
established an explicit correspondence with quartic Thue equations and
proved, in particular, that the only coincidences for which the smaller
positive parameter is at most $1000$ are
\[
 \{1,103\},
 \qquad
 \{2,22\},
 \qquad
 \{4,956\}
\]
\cite[Theorem~8.1 and Tables~2--3]{Hoshi}.
More recently, Pincus and
Washington proved several global uniqueness results: there is at most one
other positive parameter when $n\equiv2\pmod4$, and likewise when
$n\equiv8\pmod{16}$ under an additional odd-trace hypothesis on the
fundamental unit of $\Q(\sqrt{n^2+16})$; they also proved finiteness for every
fixed $n$ \cite[Theorem~1]{PW}.  The theorem below gives the complete
classification for all positive integral parameters.  In particular, it
extends Hoshi's finite-range classification and subsumes the uniqueness
conclusions of Pincus and Washington.

Our main result is the following complete classification.

\begin{theorem}\label{thm:main}
Let $m,n$ be distinct positive integers.  Then
\[
 K_m=K_n
 \quad\Longleftrightarrow\quad
 \{m,n\}\in
 \bigl\{\{1,103\},\{2,22\},\{4,956\}\bigr\}.
\]
\end{theorem}

All results from earlier papers that are used in the proof are stated in Section~\ref{sec:literature}.  Every subsequent argument refers only to those statements and to results proved in the present paper.  This has two advantages: the hypotheses used from the literature are explicit, and notation specific to the cited papers does not enter the later proof.

\section{Results used from the literature}\label{sec:literature}

This section contains no proofs.  The statements are reformulated using the notation of the present paper and only the consequences needed below.

\subsection{Results of Gras and Hoshi}

\begin{theorem}[Gras; Hoshi]\label{thm:hoshi-basic}
For every positive integer $n\neq3$, the polynomial $f_n$ is irreducible over $\Q$, and $K_n/\Q$ is a cyclic quartic extension.  Moreover
\[
 f_3(X)=(X^2+X-1)(X^2-4X-1),
 \qquad
 K_3=\Q(\sqrt5).
\]
Consequently, for every positive integer $m$,
\[
 K_m=K_3\quad\Longrightarrow\quad m=3.
\]
\end{theorem}

\noindent
The irreducibility and cyclicity assertion for $n\neq3$ is Gras's
Proposition~6 \cite[p.~16]{Gras1978}.  The factorization at $n=3$ and the
final consequence are Hoshi's Lemma~4.1 together with the discussion of the
reducible case in Section~7 \cite[Lemma~4.1 and Section~7]{Hoshi}.

\begin{theorem}[Hoshi]\label{thm:hoshi-correspondence}
Let $a,b$ be positive integers satisfying
\[
 0<a<b,
 \qquad a\neq3,
 \qquad b\neq3,
\]
and suppose that $K_a=K_b$.  Then there exist integers $x,y,c$ and $N$ such that
\begin{align}
 \gcd(x,y)&=1,
 &x&\not\equiv y\pmod2,
 &xy(x+y)(x-y)&\neq0,\label{eq:hoshi-xy}\\
 c&\mid a^2+16,
 &c&\text{ is odd},
 &F_a(x,y)&=c,\label{eq:hoshi-c}\\
 N&\in\{b,-b\}.\label{eq:hoshi-Nsign}
\end{align}
and
\begin{equation}\label{eq:hoshi-N}
 N=a+\frac{(a^2+16)xy(x+y)(x-y)}{c}.
\end{equation}
\end{theorem}

\noindent
This is the fixed-parameter consequence of Hoshi's field-isomorphism criterion and quartic-equation correspondence, with the parity refinement of Lemma~6.1 \cite[Theorems~1.1 and~1.4, Lemma~6.1]{Hoshi}.  The conditions in \eqref{eq:hoshi-xy} spell out explicitly what is needed from Hoshi's terminology.

\begin{theorem}[Hoshi]\label{thm:hoshi-finite}
Let $a,b$ be positive integers satisfying
\[
 0<a<b,
 \qquad a\le1000,
 \qquad a\neq3,
 \qquad b\neq3.
\]
If $K_a=K_b$, then
\[
 (a,b)\in\{(1,103),(2,22),(4,956)\}.
\]
\end{theorem}

\noindent
This is the consequence of Hoshi's Theorem~8.1 and Tables~2--3 obtained by retaining the entries for which the second positive parameter is larger than the first \cite[Theorem~8.1 and Tables~2--3]{Hoshi}.

\begin{theorem}[Hoshi]\label{thm:hoshi-examples}
The following equalities hold:
\[
 K_1=K_{103},
 \qquad
 K_2=K_{22},
 \qquad
 K_4=K_{956}.
\]
\end{theorem}

\noindent
The middle equality already appears in Gras's example with parameter
$t=22$, expressed there through a generating relative unit corresponding to
$t=2$ \cite[p.~18]{Gras1978}.  All three equalities are recorded in Hoshi's
Table~2 and the accompanying discussion \cite[Table~2]{Hoshi}.

\subsection{A specialized result of Lettl--Peth\H{o}--Voutier}

For $t\ge6$, let $\alpha_t$ and $\gamma_t$ denote the real roots of $F_t(X,1)$ characterized by
\begin{align}
 1-\frac2t+\frac{2}{t^2}
 &<\alpha_t<
 1-\frac2t+\frac{3}{t^2},\label{eq:alpha-def}\\
 -\frac1t+\frac{4}{t^3}
 &<\gamma_t<
 -\frac1t+\frac{5}{t^3}.\label{eq:gamma-def}
\end{align}
The intervals in \eqref{eq:alpha-def}--\eqref{eq:gamma-def} contain exactly one root each; these inequalities are part of Lemma~9(c) of Lettl--Peth\H{o}--Voutier \cite[Lemma~9(c)]{LPV}.

\begin{theorem}[Lettl--Peth\H{o}--Voutier, specialized form]\label{thm:lpv-specialized}
Let $t\ge1001$, and let $x,y\in\Z$ satisfy
\begin{equation}\label{eq:lpv-hyp}
 \gcd(x,y)=1,
 \qquad y>0,
 \qquad |x|\le y,
 \qquad |F_t(x,y)|\le t^2+16,
\end{equation}
and
\begin{equation}\label{eq:lpv-size}
 4\sqrt{\frac{t^2+16}{2t-1}}\le y.
\end{equation}
Then $x/y$ is a convergent in the simple continued-fraction expansion of either $\alpha_t$ or $\gamma_t$.

The initial partial quotients needed below are as follows.

If $t=2u$ is even, then
\begin{equation}\label{eq:alpha-cf-even}
 \alpha_t=
 \left[0;1,u-1,1,1,
 \left\lfloor\frac{u-3}{5}\right\rfloor,\ldots\right].
\end{equation}
If $t=2u-1$ is odd, then
\begin{equation}\label{eq:alpha-cf-odd}
 \alpha_t=[0;1,u-1,A_3,\ldots],
 \qquad A_3\ge\frac t3.
\end{equation}
For every $t\ge1001$,
\begin{equation}\label{eq:gamma-cf}
 \gamma_t=
 \left[-1;1,t-1,
 \left\lfloor\frac t5\right\rfloor,\ldots\right].
\end{equation}

If $y>1$, the following estimates hold.  If $x/y$ is a convergent to $\alpha_t$, then
\begin{equation}\label{eq:lpv-alpha-simple}
 y^{5/4}<\frac{10}{9}(t^2+16).
\end{equation}
If $x/y$ is a convergent to $\gamma_t$, then
\begin{equation}\label{eq:lpv-gamma-simple}
 y^{5/4}<t^2+16.
\end{equation}
\end{theorem}

\noindent
The continued fractions are Lemma~9(d) of Lettl--Peth\H{o}--Voutier.  The convergence assertion and their quantitative estimates are Theorem~2(a)--(b) of that paper \cite[Lemma~9(d), Theorem~2]{LPV}.  The bounds \eqref{eq:lpv-alpha-simple}--\eqref{eq:lpv-gamma-simple} are the weaker numerical consequences of Theorem~2(b) for $t\ge1001$ that are used in this paper; in particular, the auxiliary function appearing in the original formulation is not needed here.

\section{An identity attached to an equality of fields}\label{sec:identity}

We now begin the proof of Theorem~\ref{thm:main}.  The parameter $3$ has already been separated by Theorem~\ref{thm:hoshi-basic}, so from this point onward we apply Theorem~\ref{thm:hoshi-correspondence} only under its stated hypotheses.

\begin{proposition}\label{prop:gaussian}
Let $0<a<b$ be positive integers with $a,b\neq3$, and suppose that $K_a=K_b$.  Then there exist integers $x,y,c,N$ satisfying \eqref{eq:hoshi-xy}--\eqref{eq:hoshi-N}.  They may be chosen so that
\begin{equation}\label{eq:normalization}
 y>0,
 \qquad 0<|x|<y.
\end{equation}
For such a choice,
\begin{equation}\label{eq:normidentity}
 (a^2+16)(x^2+y^2)^4=c^2(b^2+16).
\end{equation}
\end{proposition}

\begin{proof}
Take $x,y,c,N$ from Theorem~\ref{thm:hoshi-correspondence}.  The identities
\[
 F_a(y,-x)=F_a(x,y),
 \qquad
 F_a(-x,-y)=F_a(x,y)
\]
allow us to arrange \eqref{eq:normalization}.  These transformations preserve $\gcd(x,y)$ and also preserve
\[
 xy(x+y)(x-y)=xy(x^2-y^2),
\]
so the value of $N$ in \eqref{eq:hoshi-N} is unchanged.

Put
\[
 P=xy(x^2-y^2),
 \qquad
 R=x^4-6x^2y^2+y^4.
\]
Then
\[
 c=R-aP.
\]
Using \eqref{eq:hoshi-N},
\[
 N=a+\frac{(a^2+16)P}{c}=\frac{aR+16P}{c}.
\]
Since
\[
 (x+\ii y)^4=R+4\ii P,
\]
we obtain
\begin{equation}\label{eq:gaussian}
 (4+\ii a)(x+\ii y)^4=c(4+\ii N).
\end{equation}
Taking complex norms and using $N^2=b^2$ gives \eqref{eq:normidentity}.
\end{proof}

\section{Parameters at most \texorpdfstring{$1000$}{1000}}

\begin{proposition}\label{prop:small}
Let $0<a<b$ be positive integers with $a\le1000$.  If $K_a=K_b$, then
\[
 (a,b)\in\{(1,103),(2,22),(4,956)\}.
\]
\end{proposition}

\begin{proof}
If $a=3$ or $b=3$, Theorem~\ref{thm:hoshi-basic} contradicts $a<b$.  Hence $a,b\neq3$, and the conclusion is exactly Theorem~\ref{thm:hoshi-finite}.
\end{proof}

\section{Exclusion of \texorpdfstring{$a>1000$}{a > 1000}}\label{sec:large}

\begin{proposition}\label{prop:large}
Let $a,b$ be positive integers with
\[
 1000<a<b.
\]
Then $K_a\neq K_b$.
\end{proposition}

\begin{proof}
Assume that $K_a=K_b$.  Since $a,b\neq3$, Proposition~\ref{prop:gaussian} gives integers $x,y,c,N$ satisfying \eqref{eq:normalization} and \eqref{eq:normidentity}.  Put
\[
 D=a^2+16,
 \qquad
 S=(x^2+y^2)^2.
\]
Equation \eqref{eq:normidentity} gives
\begin{equation}\label{eq:cS}
 |c|=S\sqrt{\frac{D}{b^2+16}}<S,
\end{equation}
and Theorem~\ref{thm:hoshi-correspondence} gives
\begin{equation}\label{eq:cD}
 |c|\le D.
\end{equation}

We first obtain a lower bound for $y$.  Let
\[
 \Delta=y^2-x^2>0,
 \qquad
 R=x^4-6x^2y^2+y^4.
\]
Then
\begin{equation}\label{eq:RS}
 R+S=2\Delta^2,
 \qquad
 S-R=8x^2y^2.
\end{equation}
If $x<0$, write $x=-u$ with $u>0$.  Then $P=uy\Delta>0$ and $c=R-auy\Delta$.  From $c>-S$ and \eqref{eq:RS},
\[
 auy\Delta<R+S=2\Delta^2,
\]
so
\[
 y>\frac{au}{2}\ge\frac a2.
\]
If $x>0$, then $P=-xy\Delta$ and $c=R+axy\Delta$.  From $c<S$ and \eqref{eq:RS},
\[
 axy\Delta<8x^2y^2.
\]
Writing $r=x/y\in(0,1)$ gives
\[
 a<\frac{8r}{1-r^2}
  <\frac4{1-r}
  =\frac{4y}{y-x}
  \le4y.
\]
Thus in all cases
\begin{equation}\label{eq:yquarter}
 y>\frac a4.
\end{equation}

We now verify the hypotheses of Theorem~\ref{thm:lpv-specialized} with $t=a$.  By \eqref{eq:cD},
\[
 |F_a(x,y)|=|c|\le a^2+16.
\]
Also $\gcd(x,y)=1$ and $|x|<y$ by Proposition~\ref{prop:gaussian}.  Since $a>1000$,
\[
 \frac{a^2+16}{2a-1}<a
 \qquad\text{and}\qquad
 4\sqrt a<\frac a4.
\]
Together with \eqref{eq:yquarter}, these inequalities imply
\begin{equation}\label{eq:lpv-size-large}
 4\sqrt{\frac{a^2+16}{2a-1}}<y.
\end{equation}
Therefore $x/y$ is a convergent to $\alpha_a$ or $\gamma_a$.

We first exclude the initial convergents allowed by \eqref{eq:alpha-cf-even}--\eqref{eq:gamma-cf}.  If $a=2u$, the initial convergents to $\alpha_a$ with denominator greater than $1$ are
\[
 \frac{u-1}{u},
 \qquad
 \frac{u}{u+1},
 \qquad
 \frac{2u-1}{2u+1}.
\]
If $a=2u-1$, the only initial convergent to $\alpha_a$ with denominator greater than $1$ before the term containing $A_3$ is
\[
 \frac{u-1}{u}.
\]
For $\gamma_a$, the only initial convergent with denominator greater than $1$ before the term containing $\lfloor a/5\rfloor$ is
\[
 -\frac1a.
\]
Direct substitution gives
\begin{align*}
 F_{2u}(u-1,u)&=2u^3+2u^2-4u+1,\\
 F_{2u}(u,u+1)&=-2u^3+2u^2+4u+1,\\
 F_{2u}(2u-1,2u+1)&=80u^2-4,\\
 F_{2u-1}(u-1,u)&=5u^2-5u+1,\\
 F_a(-1,a)&=1-5a^2.
\end{align*}
For $a>1000$, the absolute value of each displayed quantity is greater than $a^2+16$.  Indeed, after subtracting $a^2+16$ from the corresponding absolute value, the five differences are
\begin{gather*}
 2u^3-2u^2-4u-15,
 \qquad
 2u^3-6u^2-4u-17,
 \qquad
 76u^2-20,\\
 u^2-u-16,
 \qquad
 4a^2-17,
\end{gather*}
all of which are positive in the present range.  This contradicts \eqref{eq:cD}.  Hence $x/y$ must occur after these initial convergents.

Suppose first that $x/y$ is a later convergent to $\alpha_a$.  If $a=2u$, the denominator of the next convergent is
\[
 \left\lfloor\frac{u-3}{5}\right\rfloor(2u+1)+(u+1).
\]
Since $u$ is an integer,
\[
 \left\lfloor\frac{u-3}{5}\right\rfloor\ge\frac{u-7}{5},
\]
and hence every later denominator satisfies
\[
 y>\frac{2u^2}{5}-2u
   =\frac{a^2}{10}-a.
\]
If $a=2u-1$, \eqref{eq:alpha-cf-odd} gives
\[
 y\ge A_3u+1
 \ge\frac a3\frac{a+1}{2}+1
 >\frac{a^2}{10}-a.
\]
Thus, in either parity,
\begin{equation}\label{eq:alpha-den}
 y\ge\frac{a^2}{10}-a\ge\frac{a^2}{11}.
\end{equation}
By \eqref{eq:alpha-den},
\[
 y^{5/4}
 \ge\left(\frac{a^2}{11}\right)^{5/4}
 =\frac{a^2\sqrt a}{11^{5/4}}
 >\frac{31}{21}a^2,
\]
where $\sqrt a>31$ and $11^{5/4}<21$.  On the other hand, Theorem~\ref{thm:lpv-specialized} gives
\[
 y^{5/4}<\frac{10}{9}(a^2+16)
 <\frac{101}{90}a^2,
\]
a contradiction.

It remains to consider later convergents to $\gamma_a$.  From \eqref{eq:gamma-cf}, their denominators satisfy
\begin{equation}\label{eq:gamma-den}
 y\ge a\left\lfloor\frac a5\right\rfloor+1
 \ge\frac{a^2-4a+5}{5}
 \ge\frac{a^2}{6}.
\end{equation}
Hence
\[
 y^{5/4}
 \ge\left(\frac{a^2}{6}\right)^{5/4}
 =\frac{a^2\sqrt a}{6^{5/4}}
 >3a^2,
\]
using $\sqrt a>30$ and $6^{5/4}<10$.  But Theorem~\ref{thm:lpv-specialized} gives
\[
 y^{5/4}<a^2+16<2a^2,
\]
again a contradiction.  Thus $K_a\neq K_b$.
\end{proof}

\section{Proof of the classification}

\begin{proof}[Proof of Theorem~\ref{thm:main}]
Assume first that $K_m=K_n$ with $m\neq n$, and put
\[
 a=\min\{m,n\},
 \qquad
 b=\max\{m,n\}.
\]
Then $0<a<b$.

If $a=3$ or $b=3$, Theorem~\ref{thm:hoshi-basic} gives a contradiction.  If $a\le1000$, Proposition~\ref{prop:small} yields
\[
 (a,b)\in\{(1,103),(2,22),(4,956)\}.
\]
If $a>1000$, Proposition~\ref{prop:large} gives a contradiction.  Hence
\[
 K_m=K_n,
 \quad m\neq n
 \quad\Longrightarrow\quad
 \{m,n\}\in
 \bigl\{\{1,103\},\{2,22\},\{4,956\}\bigr\}.
\]

Conversely, the three equalities are exactly Theorem~\ref{thm:hoshi-examples}.  This proves the theorem.
\end{proof}

\begin{corollary}
For every positive integer $n$, there is at most one positive integer $m\neq n$ such that $K_m=K_n$.
\end{corollary}

\begin{corollary}
If $n$ is odd and $m\neq n$ is positive, then
\[
 K_m=K_n
 \quad\Longleftrightarrow\quad
 \{m,n\}=\{1,103\}.
\]
\end{corollary}

\section*{Declaration of AI use}

This manuscript was developed with extensive use of OpenAI's ChatGPT. ChatGPT played a major role in the selection and refinement of the research problem, the iterative formulation of the main theorem, literature searches, exploration of proof strategies, development and assembly of the proof, verification of calculations, checking the applicability and hypotheses of cited results, verification of citations and bibliographic information, and the drafting of substantially the entire manuscript. The author's role consisted primarily of reviewing, correcting where necessary, and approving the final mathematical arguments and exposition. The author takes full responsibility for the correctness and contents of the final manuscript.


\begin{thebibliography}{99}

\bibitem{Gras1978}
M.-N.~Gras,
\emph{Table num\'erique du nombre de classes et des unit\'es des
extensions cycliques r\'eelles de degr\'e~4 de $\Q$},
Publications math\'ematiques de Besan\c{c}on. Alg\`ebre et th\'eorie des
nombres, no.~2 (1978), article no.~1, 1--133.
\href{https://doi.org/10.5802/pmb.a-17}{doi:10.5802/pmb.a-17}.

\bibitem{Hoshi}
A.~Hoshi,
\emph{On the simplest quartic fields and related Thue equations},
in R.~Feng, W.-S.~Lee, and Y.~Sato (eds.), \emph{Computer Mathematics},
Springer, Berlin--Heidelberg, 2014, pp.~67--85.
\href{https://doi.org/10.1007/978-3-662-43799-5_7}{doi:10.1007/978-3-662-43799-5\_7}.

\bibitem{LPV}
G.~Lettl, A.~Peth\H{o}, and P.~M.~Voutier,
\emph{Simple families of Thue inequalities},
Trans. Amer. Math. Soc. \textbf{351} (1999), no.~5, 1871--1894.
\href{https://doi.org/10.1090/S0002-9947-99-02244-8}{doi:10.1090/S0002-9947-99-02244-8}.

\bibitem{PW}
D.~L.~Pincus and L.~C.~Washington,
\emph{On the field isomorphism problem for the family of simplest quartic fields},
Acta Arith. \textbf{218} (2025), no.~4, 347--356.
\href{https://doi.org/10.4064/aa240619-15-10}{doi:10.4064/aa240619-15-10}.

\end{thebibliography}
\end{document}